\documentclass[12pt]{article}
\usepackage{amsmath,amssymb}
\usepackage{amsthm}

\def\MC{\mathcal{M}}

\def\E{\mathbf{E}}

\def\P{\mathbf{P}}
\def\R{\mathbf{R}}

\def\1{\mathbf{1}}

\def\div{\rm{div}}

\def\al{\alpha}
\def\be{\beta}
\def\pa{\partial}

\def\de{\delta}

\newtheorem{prop}{Proposition}[section]

\newtheorem{remark}{Remark}

\newcommand{\si}{\sigma}

\newcommand{\om}{\omega}
\newcommand{\Ga}{\Gamma}
\newcommand{\De}{\Delta}

\newcommand{\Om}{\Omega}

\begin{document}
\title{Stochastic duality of Markov processes: a study via generators}
\author{Vassili Kolokoltsov and RuiXin Lee}
\maketitle

\begin{abstract}
The paper is devoted to a systematic study of the duality of processes
in the sense that $E f(X_t^x,y)=E f (x, Y_t^y)$ for a certain $f$.
 This classical topic has well known applications in interacting particles, intertwining, superprocesses, stochastic monotonicity, exit - entrance laws, ruin probabilities in finances, etc. Aiming mostly at the case of $f$ depending on the difference of its arguments,
we shall give a systematic study of duality via the analysis of the generators of dual Markov processes leading to various
 results and insights.
\end{abstract}

{\bf Key words:} stochastic monotonicity, stochastic duality, generators of dual Markov processes, reflection and absorbtion

\section{Introduction}

The paper is devoted to a systematic study of the duality of processes
in the sense that $E f(X_t^x,y)=E f (x, Y_t^y)$ for a certain $f$.
 This classical topic has well known applications in (and deep links with) interacting particles (see e.g. \cite{Lig} and references therein),
  intertwining (see e.g. \cite{Biane}, \cite{CaPeYo}, \cite{Dube}, \cite{PPTS}), superprocesses (see \cite{EK}, \cite{Myt}), stochastic monotonicity (see e.g. \cite{JMW} and \cite{Chenbook04}), exit - entrance laws (see \cite{CoRo}), ruin probabilities in finances (see \cite{Dje93}), birth and death processes
    (see \cite{VanDo}, \cite{And}) and others. Aiming mostly at the case of $f$ depending on the difference of its arguments,
we shall give a systematic study of duality via the analysis of the generators of dual Markov processes (extending the analysis of one-dimensional processes from \cite{Ko10}, \cite{Ko03}) leading to various
 results and insights.

\subsection{Objectives}

In stochastic analysis one meets various kinds of duality. For instance, the Markov processes $X_t^x$ and $Y_t^y$
with values in a Borel space $X$ are called
 dual with respect to the reference measure $\nu$ on $X$, if the duality equation
 \begin{equation}
 \label{eqdefstanddual1st}
 \int_X \E h(X_t^x) g(x) \nu (dx) =\int_X h(x) \E g(Y_x^t) \nu (dx)
\end{equation}
holds for an appropriate class of functions $h,g$,
see e.g. \cite{AnPaPa} and references therein for.

In another approach, the Markov processes $X_t^x$ and $Y_t^y$
(small $x,y$ here and in what follows stand for the initial points)
with values in possibly different Borel spaces $X$ and $Y$ are called dual with respect to a Borel function $f$ on $X\times Y$, if
 \begin{equation}
 \label{eqdefstanddual2nd}
\E f(X_t^x,y) =\E f(x,Y_t^y)
\end{equation}
for all $x\in X, y\in Y$, where $\E$ on the left hand side and the right hand side correspond to the distributions
of processes $X_t^x$ and $Y_t^y$ respectively,
see e.g. \cite{Lig} and references therein for an extensive application  of this notion in interacting particles.

A particular case of \eqref{eqdefstanddual2nd} is the duality of one-dimensional processes ($X$ and $Y$ are real-valued) arising from stochastic monotonicity, where $f(x,y)=\1_{\{x\ge y\}}$ (we denote here and in what follows by $\1_M$ the indicator function of the set $M$) and hence \eqref{eqdefstanddual2nd}
turns to the equation
\begin{equation}
 \label{eqdefstanddual2ndmon}
\P (X_t^x \ge y) =\P (Y_t^y \le x),
\end{equation}
see \cite{Sieg}. Other useful cases include $f(x,y)=e^{xy}$ or $f(x,y)=x^y$, used in particular in the theory of superprocesses, see e.g. Ch. 4 of \cite{EK} or Ch. 1 of \cite{Et}. For an application of duality in actuarial science see \cite{Dje93}.

The analytic analogs of the duality of the 1st kind is successfully used in the theory of operator semigroups independently of their probabilistic content, see e.g \cite{AnPaPa} and references therein. We shall start now with a sketch of a systematic study of duality obtained by
 extending \eqref{eqdefstanddual2nd} to general purely analytic setting aiming at the extension of the theory of
\eqref{eqdefstanddual2ndmon} to dualities generated by partial orders and more general translation invariant dualities arising from $f$ depending
 on the difference of their arguments.

There are many applications of duality in population dynamics, branching processes and other areas, see e.g. \cite{AlkHu}
and references therein.

\subsection{On the general notion of semigroup duality}

For a topological (e.g. metric) space $X$ we denote by $B(X)$ and $C(X)$ the spaces of bounded Borel measurable and bounded continuous functions respectively. Equipped with the sup-norm $\|f\|=\sup_x |f(x)|$ both these spaces become Banach spaces. Bounded signed measures on $X$ are defined as bounded $\si$-additive functions on the Borel
subsets of $X$. The set of such measures $\MC (X)$ equipped with the total variation norm is also a Banach space.
The standard duality between $B(X)$ and $\MC(X)$ is given by the integration:
\[
(f, \mu)=\int_X f(x) \mu (dx).
\]

Let $X,Y$ be two topological spaces. By a signed (stochastic) kernel from $X$ to $Y$ we mean a function of two variables $p(x, A)$, where $x\in X$ and $A$ are Borel subsets of $Y$ such that $p(x,.)$ is a bounded signed measure on $Y$ for any $x$ and $p(., A)$ is a Borel function for any Borel set $A$. We say that this kernel is bounded if $\sup_x \|p(x,.)\|<\infty$. We say that this kernel is weakly continuous if the mapping $x\mapsto p(x,.)$ is continuous with measures $\MC(Y)$ considered in their weak topology.
If all measures $p(x,.)$ are positive, the corresponding kernel is called a stochastic kernel.

Any bounded kernel specifies a bounded linear operator $B(Y) \to B(X)$ via the formula
\[
Tg(x)=\int_Y g(z) p(x, dz).
\]
$T$ is said to be the integral operator with the kernel $p$.
The standard dual operator $T'$ is defined as the operator $\MC(X)\to \MC(Y)$ specified by the duality relation
\[
(f,T'\mu)=(Tf, \mu),
\]
or explicitly as
\[
T'\mu(dy)=\int_X p(x, dy) \mu (dx).
\]

Clearly the kernel $p(x,dz)$ is weakly continuous if and only if $T$ acts on continuous functions, that is, $T :C(Y) \to C(X)$.

Let $p(x,dz)$ be a bounded signed kernel from $X$ to itself, $T$ the corresponding integral operator, and let $f(x,y)$
 be a bounded measurable function on $X\times Y$. Let us say that the operator $T^{D(f)}: B(Y)\to B(T)$ is $f$-{\it dual} to $T$, if
\begin{equation}
\label{eqdefdualsemigroups}
(T^{D(f)}f(x,.))(y)=(Tf(., y))(x)
\end{equation}
for any $x,y$, that is, the application of $T^D$ to the second argument of $f$ is equivalent to the application of $T$ to its first argument.
Of course, if $T^{D(f)}$ is $f$-dual to $T$, then $T$ is $\tilde f$-dual to $T^{D(f)}$ with $\tilde f(y,x)=f(x,y)$.

We say that $f$ {\it separates points of} $X$ if,
for any $x_1,x_2 \in X$, there exists $y \in Y$ such that $f(x_1, y)\neq f(x_2,y)$. The following is a bit more nontrivial notion. We say that
 $f$ {\it separates measures on} $X$ if,
for any $Q_1,Q_2 \in \MC(X)$, there exists $y \in Y$ such that $\int f(x,y) Q_1(dx)\neq \int f(x,y) Q_2(dx)$.
  If this is the case, the integral operator $F=F_f:\MC(X)\to B(Y)$ given by
\begin{equation}
\label{eqdefdualsemigroupsF}
(FQ)(y)=\int f(x,y) Q(dx)
\end{equation}
is an injective bounded operator, so that the linear inverse $F^{-1}$ is defined on the image $F(\MC(X))$.
Let us say that the function $FQ$ is $f$-generated by $Q$.

\begin{remark}
In \cite{CoRo} the authors call a function $g$ to be representable by $f$, if there exists a unique $Q$ such that $g=FQ$.
Paper \cite{CoRo} deals with the application of duality to exit and entrance laws of Markov processes.
\end{remark}

\subsection{Basic tools}

\begin{prop}
\label{propdualop}
Let $f$ be a bounded measurable function separating measures on $X$ and $T$ an integral operator in $B(X)$ with a bounded signed kernel $p$.
Then $T^{D(f)}$ is well defined on $F(\MC(X))$ and its action on the $f$-generated functions coincides with $T'$, that is
\begin{equation}
\label{eqpropdualop}
T^{D(f)}=F \circ T' \circ F^{-1},
\end{equation}
or equivalently
\begin{equation}
\label{eqpropdualop1}
F^{-1} \circ T^{D(f)} = T' \circ F^{-1}.
\end{equation}
In other words, the $f$-dual operator $T^{D(f)}$ is obtained by the 'dressing' of the standard dual $T'$ by the operator $F$.
\end{prop}

\begin{proof}
Let $g\in F(\MC(X))$ be given by $g(y)=\int f(x,y) Q_g(dx)$. Then
\[
 T^{D(f)}g(y)=\int_X (T^{D(f)}f(x,.))(y) Q_g(dx)
\]
\[
=\int_X (Tf(.,y))(x) Q_g(dx)=\int_X\int_Y f(z,y) p(x,dz)Q_g(dx)=\int_Y f(z,y) \tilde Q (dz),
\]
with
\[
\tilde Q(dz)=\int p(x,dz) Q_g(dx).
\]
Thus $T^{D(f)}g$ is $f$-generated by $\tilde Q=T'Q_g$, as required.
\end{proof}

\begin{remark}
Equation \eqref{eqpropdualop1} is a particular case of the so-called intertwining, see \cite{Biane}, \cite{CaPeYo}, as well as
\cite{Dube}, \cite{PPTS}, \cite{HiYo} for exciting recent developments.
Relations \eqref{eqpropdualop1} for discrete Markov chains are analyzed in detail in \cite{HuiMar}.
\end{remark}

Representation \eqref{eqpropdualop} has a direct implication for the theory of semigroups.

\begin{prop}
\label{propdualsemi}
Let $f$ be a bounded measurable function separating measures on $X$ and $T_t$ a semigroup of integral operators in $B(X)$ specified by the family of bounded signed kernel $p_t(x,dz)$ from $X$ to $X$, so that $T_0$ is the identity operator and $T_tT_s=T_{t+s}$, which, in terms of kernels, rewrites as the Chapman-Kolmogorov equation
\[
\int_X p_t(x,dz) p_s (z,dw)=p_{t+s}(x, dw).
\]
Then the dual operators $T_t^{D(f)}$ in $F(\MC(X))$ also form a semigroup, so that
\begin{equation}
\label{eqpropdualopesemig}
T_t^{D(f)}=F \circ T'_t \circ F^{-1}.
\end{equation}
\end{prop}

\begin{proof}
This is straightforward from \eqref{eqpropdualop} and the standard obvious fact that $T_t'$ form a semigroup in $\MC(X)$.
\end{proof}

\begin{remark}
The duality \eqref{eqdefstanddual1st} is of course also included in the general scheme above, that is, the dual can still be expressed
as \eqref{eqpropdualop}. For instance, if $\nu(dx)$ has a density
 $\nu(x)$ with respect to Lebesgue measure and $T'$ can be reduced to the action on functions, then $F^{-1}$ is the multiplication on $nu(x)$
 and $f(x,y)=\de (x-y) \nu^{-1}(x)$.
\end{remark}

It is also worth noting that the assumption of boundedness of $f$ is not very essential.
 If it is not bounded (and we shall discuss interesting examples of such situations later), the integral operator $F$ will not be defined on all bounded measures, but only on its subspace. This will be reflected in the domain of $T^{D(f)}$, but the whole scheme of Proposition \ref{propdualop} still remains valid.

\subsection{Links with differential equations and stochastic processes}

Let us explain briefly the main ideas on the application of the above results to the theory of differential equations and stochastic processes. Precise details for concrete situations will be discussed below.

Let a semigroup $T_t$ in $B(X)$ be generated by a (possibly unbounded) operator $L$ in $B(X)$ defined on an invariant (under all $T_t$) domain $D\subset B(X)$, so that
\[
\left.\frac{d}{dt}\right|_{t=0}T_th=\lim_{t\to 0} \frac{1}{t} (T_th-h)=Lh, \quad h\in D,
\]
with convergence in some appropriate topology (say, strongly or point-wise) and thus the operators $T_t$ represent resolving operators for the Cauchy problem of the equation
$\dot h =Lh$. Then \eqref{eqpropdualop} implies that
\[
\left.\frac{d}{dt}\right|_{t=0}T_t^{D(f)}g= F \circ \left.\frac{d}{dt}\right|_{t=0} T' \circ F^{-1}g=F \circ L' \circ F^{-1}g,
\]
that is, the generator of the semigroup $T_t^{D(f)}$ is
\begin{equation}
\label{eqdualgengen}
L^{D(f)}=F \circ L' \circ F^{-1},
\end{equation}
 so that the operators $T_t^{D(f)}$ represent resolving operators for the Cauchy problem of the equation
$\dot g =L^{D(f)}g$. Here $L'$ is of course the standard dual operator to $L$.
Thus duality can yield explicit solutions for equations of this kind. Of course, our arguments were heuristic as we did not pay attention to the domain of definition of $L'$, which should be done
in concrete situations. The main difficulty here is to characterize the operator $F_f$.

Next, in order to be able to fill the duality equation \eqref{eqdefdualsemigroups} with probabilistic content, i.e. to rewrite it as \eqref{eqdefstanddual2nd}, the semigroups $T_t$ and $T_T^{D(f)}$ should be positivity preserving and generate some Markov processes.

This question effectively reduces to the question of whether, for a given conditionally positive operator $L$, the corresponding
dual $L^{D(f)}$ is also conditionally positive.

It is seen now that the basic questions to be addressed to make the theory work for concrete functions $f$ are (i) the characterization of the operators $F$ and $F^{-1}$
(for the analytic part of the story) and (ii) the criteria for conditional positivity of $L^{D(f)}$ (for its probabilistic content).

As we shall see it is often convenient to reduce the operator $F$ to some subclass of Borel measures $Q$, where its inverse can be explicitly found.
For instance, it is often easier to work with $Q$ having density with respect to some
reference measure.

\subsection{Content and plan of the paper}

We shall apply formulas \eqref{eqpropdualopesemig} and \eqref{eqdualgengen} to characterize classes of dual Markov processes with respect to
various functions $f$ depending on the difference of its arguments.  Section \ref{secPardu} deals with duality on $\R^d$ arising from Pareto
and similar partial orders. The full characterization of duality is given in terms of generators for basic classes of Feller processes.
Section \ref{secdualfromLevygen} discusses several examples of duality with operator $F^{-1}$ being the Laplacian or a fractional Lapacian. Section \ref{secdualbound} initiates an application of formulas \eqref{eqpropdualopesemig} and \eqref{eqdualgengen} to the study of duality for processes in domains with a boundary. To circumvent specific difficulties arising from the boundary, we introduce here an additional tool of a regularized dual.

The extension of the theory to time-nonhomogeneous Markov processes will be analyzed in \cite{Kolopre}.
 
\section{Duality from orders and other binary relations}
\label{secPardu}

\subsection{Basic notions}

As our basic example we consider $f$-duality for functions $f$ arising from translation-invariant partial orders, or more generally, from translation-invariant
binary relations. Namely, let $X$ be a topological linear space and $M$ a Borel subset of $X$. Then $M$ defines a translation-invariant
 binary relation $R_M$ on $X$ such that
$xR_My$ means, by definition, that $x-y\in M$, or $x\in y+M$.

 Let $\tilde M=\{(x,y)\in X\times X: xR_My\}$. Let us say that the duality \eqref{eqdefdualsemigroups} arises from
 the binary relation $M$, if
\begin{equation}
\label{eqfdualbinary}
f(x,y)=f_M(x,y)=\1_{\tilde M}(x,y)=\1_{x-M}(y)=\1_{y+M}(x).
\end{equation}

\begin{remark}
If $f$-duality arises from a translation-invariant binary relation $R_M$ and if both $T_t$ and $T_t^{D(f)}$
 are known to be integral operators with kernels $p_t(x,dz)$ and $p_t^{D(f)}(y,dw)$ respectively,
 one can give another instructive proof
of Proposition \ref{propdualsemi} bypassing representation \eqref{eqpropdualop} and using instead Fubbini's theorem, as was done
in \cite{Sieg} for standard one-dimensional duality. Namely, it is sufficient to show the semigroup identity $T_{t+s}^{D(f)}=T_s^{D(f)}T_t^{D(f)}$ applied to the functions
$f(x,.)=\1_{x-M}$, as it then extends to the whole $F(\MC(X))$ by linearity. And for these functions we have
\[
(T_{t+s}^{D(f)}\1_{x-M})(y)=(T_{t+s}\1_{y+M})(x)=(T_t(T_s\1_{y+M}))(x)=\int p_t(x,dz)(T_s\1_{y+M})(z)
\]
\[
=\int p_t(x,dz)(T^{D(f)}_s\1_{z-M})(y)
=\int p_t(x,dz) \left( \int \1_{z-M}(w) p_s^{D(f)}(y,dw)\right).
\]
Applying Fubbini's theorem this rewrites as
\[
\int p_s^{D(f)}(y,dw) \int \1_{w+M}(z) p_t(x,dz)=\int (T_t  \1_{w+M})(x) p_s^{D(f)}(y,dw)
=T_s^{D(f)} (T_t^{D(f)} \1_{x-M})(y),
\]
as required.
\end{remark}

 If $M$ contains the origin and is closed under the addition of vectors, then the relation
$R_M$ is a pre-order (i.e. it is reflexive and transient) and can be naturally denoted by $\ge_M$.
If this is the case and $T_t$ and $T_t^{D(f)}$ are integral operators with positive stochastic kernels
thus specifying Markov processes, then duality relation \eqref{eqdefdualsemigroups} or equivalently
\eqref{eqdefstanddual2nd} turns to the equation
\begin{equation}
\label{eqdefdualfrompreorder}
\P (X_t^x \ge_M y) =\P (Y_t^y \le_M x),
\end{equation}
extending one-dimensional duality \eqref{eqdefstanddual2ndmon}.

The basic example we are going to analyze now is the Pareto partial order in $X=\R^d$, i.e. $\ge_M$ with $M=\R^d_+$, and its natural extension
with $M=C(e_1, \cdots, e_d)$ the cone generated by $d$ linear independent vectors $\{e_1, \cdots, e_d\}$ in $\R^d$:
\begin{equation}
\label{eqdefconepolyhedron}
C(e_1, \cdots, e_d)=\{x=\sum_{j=1}^d \al_j e_j: \quad \al_j \ge 0, \, j=1, \cdots, d\}.
\end{equation}
Of course the relation $\ge_M$ with such $M$ is again a Pareto order, but in a transformed system of coordinates.

Let us start with $M=\R^d_+$  corresponding to the Pareto order, which we shall denote just by $\ge$ omitting the subscript $M$.
The corresponding dual semigroups or processes (if exist) will be referred to as Pareto dual.
 In this case
 \begin{equation}
\label{eqParetoF}
(FQ)(y)=\int f_M(x,y) Q(dx)=\int_{x\ge y} Q(dx)
\end{equation}
is just the usual multidimensional distribution function for the measure $Q$ on $\R^d$. It is known (and easy to see) that $FQ$ characterizes $Q$ uniquely implying
that $F$ is injective and thus
 $f_M$ separates measures on $\R^d$ yielding the main condition of Proposition \ref{propdualop}.
 Moreover, if $Q$ has a density $q$ with respect to Lebesgue measure, then $q$ can be found from $FQ=g$ by differentiation:
  \begin{equation}
\label{eqParetoFinv}
 q(y_1,\cdots, y_d)=F^{-1}g(y)=(-1)^d\frac{\pa^d g(y)}{\pa y_1 \cdots \pa y_d}.
 \end{equation}
 Thus, for the Pareto order, the operator $F^{-1}$ has the simple explicit expression.

 In the case of the orders arising from the cones $M=C(e_1, \cdots, e_d)$ given by
 \eqref{eqdefconepolyhedron} this formula generalizes to
 \begin{equation}
\label{eqdensityofdistributioncon}
 q(y_1,\cdots, y_d)=(F^{-1}g)(y)
 =(-1)^d\frac{\frac{\pa ^d g}{\pa y^d}(y)[e_1,e_2, \cdots, e_d]}{|\det (e_1, e_2, \cdots, e_d)|},
 \end{equation}
 where $\det (e_1, e_2, \cdots, e_d)=\det (e_i^j)$ is the determinant of the matrix whose $i$th columns consist of the coordinates of the vector $e_i$ and
 \[
 \frac{\pa ^d g}{\pa y^d}(y)[e_1,e_2, \cdots, e_d]
 =\sum_{i_1, i_2, \cdots , i_d}
 \frac{\pa ^d g}{\pa y_{i_1} \cdots \pa y_{i_d}}(y) e_1^{i_1}e_2^{i_2} \cdots e_d^{i_d}.
 \]

 \begin{remark} For completeness, let us sketch a proof of this formula.
  If a measure $Q$ on $\R^d$ has a continuous density $q$, so that
 \[
 g(x)=FQ(x)=\int_{y+C(e_1, \cdots, e_d)}q(z) dz,
 \]
the function  $q$ can be clearly found as the limit
 \begin{equation}
\label{eqdensityofdistributionfun}
 q(y) =\lim_{h\to 0}\int_{y+\Pi (he_1, \cdots, he_d)} q(z) \, dz |\Pi (he_1, \cdots, he_d)|^{-1},
 \end{equation}
 where
 \[
 \Pi (he_1, \cdots, he_d)=\{x=\sum_j \al_j he_j, \quad \al_j\in [0,1]\}
 \]
 is the parallelepiped built on the vectors $\{he_1, \cdots, he_d\}$ and
 \[
 |\Pi (he_1, \cdots, he_d)|=h^d |\det (e_i^j)|
 \]
 is its Euclidean volume.

 From simple combinatorics it follows (see e.g. \cite{Kallen}) that
 \[
 \int_{y+\Pi (he_1, \cdots, he_d)} q(z) \, dz
 \]
 \[
 =g(y)-\sum_j g(y+he_j)+\sum_{i<j} g(y+he_i+he_j)
 +\cdots +(-1)^d g(y+he_1+\cdots +he_d).
 \]
Let us expand all terms in Taylor series up to the derivatives of order $d$. As the final expression should be of order
$h^d$ (to get a limit in \eqref{eqdensityofdistributionfun}) we conclude that all terms with the derivatives of orders less than $d$ necessarily cancel,
 so that
\[
\int_{y+\Pi (he_1, \cdots, he_d)} q(z) \, dz
\]
 \begin{equation}
\label{eqdensityofdistributionfun1}
=\frac{1}{d!} h^d \left(-\sum_j \frac{\pa ^d g}{\pa y^d}[e_j]
+\sum_{i<j} \frac{\pa ^d g}{\pa y^d}[e_i+e_j]
+\cdots + (-1)^d \frac{\pa ^d g}{\pa y^d}[e_1+\cdots +e_d]\right) +O(h^{d+1}),
\end{equation}
where $O(h^{d+1})$ denotes the expression of order $h^{d+1}$ that does not contribute to the limit in
 \eqref{eqdensityofdistributionfun}, and where we use the well established (though a bit ambiguous)
 notation for the action of the higher order derivative on equal vectors:
 \[
 \frac{\pa ^d g}{\pa y^d}(y)[v]=\frac{\pa ^d g}{\pa y^d}(y)[v, \cdots, v].
\]
It remains to note that all terms in expansion \eqref{eqdensityofdistributionfun1} containing products of coordinates
of coinciding vectors should vanish (otherwise, using different scaling on $e_i$ we would arrive to a contradiction with the existence of the limit in \eqref{eqdensityofdistributionfun}). The only non-vanishing terms should contain the products
of $d$ coordinates of all $d$ vectors. All these products comes from the last term
in the sum \eqref{eqdensityofdistributionfun1} leading to \eqref{eqdensityofdistributioncon}.
\end{remark}

For instance, let us consider a 'two-dimensional light cone'
\begin{equation}
\label{eqdeftwodimlightcone}
C(e_1,e_2)=\{(x,y): y\ge |x|\} \in \R^2,
\end{equation}
corresponding to vectors $e_1=(1,1), e_2=(-1,1)$.
Then formula \eqref{eqdensityofdistributioncon} for the inverse operator turns to the simple wave operator
\begin{equation}
\label{eqdensityofdistributionlightcon2dim}
 q(x, y)=F^{-1}g(x,y)
 =\frac12 \left(\frac{\pa ^2 g}{\pa y^2}-\frac{\pa ^2 g}{\pa x^2}\right) (x,y).
 \end{equation}

 \subsection{Duality from Pareto order: global analysis}

Let us now make the detailed analysis of the duality arising from the standard Pareto order in $\R^d$, i.e. with
$M=\R^d_+$. We aim at (i) finding explicitly the dual operator $L^{D(f)}$
for the main classes of the generators of Feller processes in $\R^d$ including diffusions and jump processes and
(ii) establishing criteria (in terms of the initial operator $L$) ensuring that this dual operator is conditionally
positive and specifies a Markov process, so that the duality relation \eqref{eqdefdualfrompreorder} holds that we shall
write simply as
\begin{equation}
\label{eqdefdualfromParetoorder}
\P (X_t^x \ge y) =\P (Y_t^y \le x)
\end{equation}
for the case of the Pareto partial order.

Let us analyze formula
\eqref{eqpropdualop} from Proposition \ref{propdualop}.
In the case of duality arising from Pareto order and the operator $T$ being integral with a probability kernel $p(x,dz)$ (i.e. all measures $p(x,.)$ are probability measures, as is the case for transition operators of Markov processes) it states that for a distribution function $g$ of a measure $Q$ on $\R^d$. i.e.
$g(x)=\int_{z\ge x} Q(dz)$ we have
\begin{equation}
\label{eqpropdualopPar1}
T^{D(f)}g(x)=F \circ T' \circ F^{-1}g(x)=\int_{y\ge x} \int_{\R^d}p(z,dy) Q(dz).
\end{equation}
We are interested in the question of when this operator can be extended to all bounded measurable $g$ as a positive operator preserving constants, i.e. as an integral operator with a probability kernel.

Assume first that the measure $Q$ has a continuous density $q$ so that \eqref{eqParetoFinv} holds, i.e.
\[
q(x)=(-1)^d \frac{\pa g^d}{\pa x_1 \cdots \pa x_d}.
\]
In this case
\begin{equation}
\label{eqpropdualopPar2}
T^{D(f)}g(x)=(-1)^d \int_{y\ge x} \int_{\R^d} p(z,dy) \frac{\pa g^d}{\pa z_1 \cdots \pa z_d} dz.
\end{equation}

We like to get rid of the derivatives of $g$. To be able to do it, let us assume that the kernel $p(x,dz)$ is weakly continuous and has weakly
continuous mixed derivatives, that is, for any $I\subset \{1, \cdots, d\}$ (including $\{1, \cdots, d\}$ itself) the mixed derivative
\begin{equation}
\label{eqD1}
\frac{\pa p^{|I|}}{\pa z_I} (z, dy)
\end{equation}
is a well defined weakly continuous kernel (possibly signed).
Then, integrating the integral over $z$ in \eqref{eqpropdualopPar2} by parts $d$ times and
assuming that all boundary terms vanish, we get
\begin{equation}
\label{eqpropdualopPar3}
T^{D(f)}g(x)= \int_{\R^d} \left(  g(z) \int_{y\ge x} \frac{\pa p^d}{\pa z_1 \cdots \pa z_d} (z, dy) \right) dz.
\end{equation}
This is an integral operator with the integral kernel (more precisely its density)
\[
p^D(x,z)=\int_{y\ge x} \frac{\pa p^d}{\pa z_1 \cdots \pa z_d} (z, dy).
\]
For this operator to be positive and constant preserving, necessary conditions are that, for all $x\in \R^d$,
\begin{equation}
\label{eqD2}
\int_{y\ge x} \frac{\pa p^d}{\pa z_1 \cdots \pa z_d} (z, dy)\ge 0,
\end{equation}
\begin{equation}
\label{eqD3}
\int\left(  \int_{y\ge x} \frac{\pa p^d}{\pa z_1 \cdots \pa z_d} (z, dy)\right) dz =1.
\end{equation}

From the integration by parts it is seen that for the last condition to hold it is sufficient to assume that
for any subset $I\subset \{1, \cdots, d\}$ excluding the whole set $\{1, \cdots, d\}$,
\begin{equation}
\label{eqD4}
\lim _{z_{\bar I} \to -\infty} \int_{\R^{|I|}} dz_I  \int_{y\ge x} \frac{\pa p^{|I|}}{\pa z_I} (z_I,z_{\bar I}, dy)=0,
\end{equation}
and there exists a finite limit
\begin{equation}
\label{eqD5}
\lim _{z_{\bar I} \to \infty} \int_{\R^{|I|}} dz_I  \int_{y\ge x} \frac{\pa p^{|I|}}{\pa z_I} (z_I,z_{\bar I}, dy),
\end{equation}
which equals $1$ for the empty set $I$.
Moreover, one sees by inspection that this condition also ensures that integrating by parts \eqref{eqpropdualopPar2}
for a $g$ having finite density \eqref{eqParetoFinv}, all boundary terms will in fact vanish, justifying
equation \eqref{eqpropdualopPar3}.

Thus we have proved the following statement.

\begin{prop}
\label{propglobalParetodualan}
Suppose an integral operator $T$ in $B(\R^d)$ is given by a probability kernel $p(x,dy)$ having
all mixed derivatives \eqref{eqD1}  well defined and weakly continuous and such that
\eqref{eqD2} holds, \eqref{eqD4} holds for any subset $I\subset \{1, \cdots, d\}$ excluding the whole set $\{1, \cdots, d\}$,
 and there exists a finite limit
\eqref{eqD5}, which equals $1$ for the empty set $I$. Then the Pareto dual operator $T^{D(f)}$ is also an integral operator
with a probability kernel.
\end{prop}

Condition \eqref{eqD2} is of course not directly verifiable. Therefore we shall see how it can be read from the generator
of the process.

\subsection{Duality from Pareto order: deterministic and diffusion processes}

We plan now to find the generators of the dual processes, when they exist.
  Let us start with the simplest case of deterministic processes generated by the first order differential operators of the form
  \begin{equation}
  \label{eqdetermgen}
  L\phi(x) =(b(x), \nabla \phi (x))=\sum_{j=1}^d b_j(x) \frac{\pa \phi}{\pa x_j}.
  \end{equation}

  In this case the dual operator is well defined on functions and
  \[
  L'g(x)=-{\div} (gb)(x)=-\sum_j \frac{\pa}{\pa x_j}[b_j(x) g(x)].
  \]

For a vector $x=(x_1, \cdots, x_d) \in \R^d$ let us denote $\check {x}_i$ the vector in $\R^{d-1}$ obtained from $x$ by deleting the coordinate $x_i$.
For a function $g(x)$ let us write $g(\check {z}_i, x_i)$ for the value of $g$ on the vector, whose $i$th coordinate is $x_i$, and other
coordinates are those of the vector $z$. Let us write $d\check {z}_j$ for the product of differentials $dz_k$ with all $k=1, \cdots, d$ excluding $j$.

Integrating by parts and assuming that $g$ decays quickly enough so that the boundary terms at infinity vanish, we have
 \[
 L^{D(f)}g(x)= FL' F^{-1}g(x)=(-1)^{d+1}\int_{z\ge x} \sum_j \frac{\pa}{\pa z_j}\left[b_j(z) \frac{\pa^d g(z)}{\pa z_1 \cdots \pa z_d}\right]\, dz_1\cdots dz_d
 \]
 \begin{equation}
  \label{eqdetermgendual}
 =(-1)^d\sum_j \int_{\check{z}_j \ge \check{x}_j}b_j(\check {z}_j, x_j)  \frac{\pa^d g(z)}{\pa z_1 \cdots \pa z_d} (\check {z}_j, x_j)\, d\check {z}_j.
 \end{equation}

 In general one cannot simplify this expression much further, and this is not a conditionally positive operator
 (it does not have a L\'evy-Khintchin form with variable coefficients) without further assumptions.

\begin{prop}
\label{propParetodualdeterm}
Let $L$ have form \eqref{eqdetermgen} with all $b_j \in C^1(\R^d)$ (the space of bounded continuous functions with
bounded continuous derivatives). Then $L^{D(f)}$ is given by
\eqref{eqdetermgendual}, so that the solution to the Cauchy problem of the equation $\dot g=L^{D(f)}g$ is given
by the corresponding formula \eqref{eqpropdualop} with $F$ and $F^{-1}$ given by \eqref{eqParetoF} and \eqref{eqParetoFinv}.
Moreover, if  each $b_j$ depends only on the coordinate $x_j$, then
 \begin{equation}
 \label{eqourdualdeterm}
  L^{D(f)}g(x)=-b_j(x_j)\frac{\pa g}{\pa x_j},
  \end{equation}
  that is, $L^{D(f)}$ coincides with $L$ up to a sign and the dual process exists and is just the deterministic motion in the opposite direction to the original one.
\end{prop}

\begin{proof}
Formula \eqref{eqourdualdeterm} is straightforward from \eqref{eqdetermgendual} and the assumptions made on $b_j$. This makes the last statement
plausible. However, strictly speaking, having the generator calculated on some subclass of functions does not directly imply that the semigroup $T^{D(f)}$
coincides with the semigroups on $C(\R^d)$ generated by operator \eqref{eqourdualdeterm}. The simplest way to see that this is in fact the case is via durect calculations with the semigroup $T_t^{D(f)}$ itself. Namely, if the deterministic Markov process $X_t^x$ with generator \eqref{eqdetermgen} can be expressed as $X_t^x =X^t(x)$ via the solutions $X^t(x)$ of the Cauchy problem for the ODE $\dot x=b(x)$, its transition kernel takes the form
$p_t(z, dy)=\de (y-X^t(z))$. Then \eqref{eqpropdualopPar2} becomes

\begin{equation}
\label{eqpropdualopPar2det}
T_t^{D(f)}g(x)=(-1)^d \int_{X^t(z)\ge x} \frac{\pa g^d}{\pa z_1 \cdots \pa z_d} dz.
\end{equation}
Under the assumption that $b_i$ depend only on $x_i$,
the coordinates of $X^t(z)$ are themselves solutions $X_i^t(z_i)$ of the one-dimensional ODE $\dot x_i =b_i(x_i)$, so that one has
\begin{equation}
\label{eqpropdualopPar2det1}
T_t^{D(f)}g(x)=(-1)^d \int_{X^t_i(z_i)\ge x_i} \frac{\pa g^d}{\pa z_1 \cdots \pa z_d} dz.
\end{equation}
From the obvious monotonicity of one-dimensional ODE this rewrites as
\begin{equation}
\label{eqpropdualopPar2det2}
T^{D(f)}g(x)=(-1)^d \int_{z_i\ge (X^t_i)^{-1} (x_i)} \frac{\pa g^d}{\pa z_1 \cdots \pa z_d} dz
=g((X^t)^{-1}(x)),
\end{equation}
which is of course the semigroup generated by the operator \eqref{eqourdualdeterm}.
\end{proof}

  Let us turn to a diffusion operator having the form
   \begin{equation}
 \label{eqdiffgen}
  L\phi(x)=(a(x)\nabla, \nabla)\phi(x)=\sum_{i,j=1}^d a_{ij}(x) \frac{\pa^2 \phi}{\pa x_i \pa x_j}(x)
  \end{equation}
  with a positive definite diffusion matrix $a(x)=(a_{ij}(x))$.

  In this case
  \[
  L'g(x)=\sum_{i,j=1}^d \frac{\pa^2}{\pa x_i \pa x_j}[a_{ij}(x)g(x)],
  \]
  and consequently
  \[
 L^{D(f)}g(x)= FL' F^{-1}g(x)=(-1)^{d}\int_{z\ge x} \sum_{i,j=1}^d \frac{\pa^2}{\pa z_i \pa z_j}
 \left[a_{ij}(z) \frac{\pa^d g(z)}{\pa z_1 \cdots \pa z_d})\right] \, dz_1\cdots dz_d.
 \]
 Let us integrate twice by parts the terms containing mixed derivatives and integrate once by parts the remaining terms. This yields
  \[
 L^{D(f)}g(x)= (-1)^{d-1}\sum_{j=1}^d \int_{\check{z}_j \ge \check{x}_j}\frac{\pa }{\pa x_j}
  \left[a_{jj}(\check {z}_j, x_j) \frac{\pa^d g(z)}{\pa z_1 \cdots \pa z_d}(\check {z}_j, x_j)\right]\, d\check {z}_j
\]
\[
+ 2(-1)^d \sum_{i<j} \int_{\check{z}_{ij} \ge \check{x}_{ij}}
 \left[a_{ij} \frac{\pa^d g}{\pa z_1 \cdots \pa z_d}\right](\check{z}_{ij}, x_i,x_j)\, d\check {z}_{ij},
 \]
 where $\check {z}_{ij}$ denotes the vector in $\R^{d-2}$ obtained from $z$ by deleting $i$th and $j$th coordinates, and $(\check{z}_{ij}, x_i,x_j)$ is the vector
 with $i$th and $j$th coordinates taken from the vector $x$, and other coordinates taken from the vector $z$.
In case $d=1$, the second sum in this expression is of course empty.

  Again in general case one cannot simplify this expression essentially.
However, assuming additionally that the coefficients $a_{ij}$ depends only on the coordinates $x_i,x_j$ (in particular, $a_{ii}$ depends only on $x_i$), we have
  \[
 L^{D(f)}g(x)= (-1)^{d-1}\sum_{j=1}^d \int_{\check{z}_j \ge \check{x}_j}\frac{\pa }{\pa x_j}
  \left[a_{jj}(x_j) \frac{\pa^d g(z)}{\pa z_1 \cdots \pa z_d}(\check {z}_j, x_j)\right] \, d\check {z}_j
\]
\[
+ 2(-1)^d \sum_{i<j} \int_{\check{z}_{ij} \ge \check{x}_{ij}}
 a_{ij} (x_i,x_j)\frac{\pa^d g}{\pa z_1 \cdots \pa z_d}(\check{z}_{ij}, x_i,x_j)\, d\check {z}_{ij}.
 \]
 Integrating by parts with respect to the variables $\check {z}_j$ in the first sum and the variables $\check {z}_{ij}$ in the second, yields
 (assuming the boundary terms at infinity vanish)
 \begin{equation}
 \label{eqourdualdiff}
  L^{D(f)}g(x)=\sum_{j=1}^d \frac{\pa }{\pa x_j}
  \left[a_{jj}(x_j) \frac{\pa g(x)}{\pa x_j}\right]
+ 2 \sum_{i<j} a_{ij} (x_i,x_j)\frac{\pa^2 g}{\pa x_i \pa x_j},
\end{equation}
or
 \begin{equation}
 \label{eqourdualdiff1}
  L^{D(f)}g(x)=Lg(x)+\sum_{j=1}^d \frac{\pa a_{jj}(x_j)}{\pa x_j}
   \frac{\pa g(x)}{\pa x_j}.
\end{equation}

\begin{prop}
\label{propParetodualdiff}
Let $L$ have form \eqref{eqdiffgen} with a positive definite diffusion matrix $a(x)=(a_{ij}(x))$ and
with all $a_{ij} \in C^1(\R^d)$, so that $L$ generates a Feller diffusion
in $\R^d$ that we denote $X_t^x$.
If the coefficients $a_{ij}$ depends only on the coordinates $x_i,x_j$,
then $L^{D(f)}$ is given by
\eqref{eqourdualdiff1} and it also generates a diffusion process in $\R^d$ that we denote $Y_t^y$,
and the duality relation
\eqref{eqdefdualfromParetoorder} holds.
 \end{prop}

\begin{proof}
Again formula \eqref{eqourdualdiff1} makes the statement very plausible, but to deduce \eqref{eqpropdualop}  from \eqref{eqdualgengen}
additional argument is of course needed. This goes as follows.

But notice first that it is sufficient to prove the statement under additional assumption that coefficients $a_{ij}$ are infinitely smooth
with all derivatives bounded
(actually we need twice differentiability for the above calculation of $L^{D(f)}$ and $d$ times differentiability for the formulas of Proposition
\ref{propglobalParetodualan} to make sense) and the operator $L$ is strictly elliptic, because any $L$ of type \eqref{eqdiffgen}
can be approximated by the sequence of $L$ of the same form but strictly elliptic and with
smooth coefficients. Passing to the limit in the duality equation allows one to prove its validity for the general case.

Next, under this smoothness and non-degeneracy assumption, it is well known from the standard theory of diffusions (or Ito's processes)
that operator \eqref{eqourdualdiff1} generates a unique Feller process such that its semigroup $T_t^{D(f)}$ preserves the space
$C^2_{\infty}(\R^d)$ of twice continuously differentiable functions vanishing at infinity with all its derivatives up to order two.
Hence, the Cauchy problem for the equation
\[
\dot g=L^{D(f)}g
\]
is well posed in classical sense for initial functions $g_0$ from $C^2_{\infty}(\R^d)$.
It is then straightforward to see  \eqref{eqdualgengen} that both functions $T_t^{D(f)}g_0$ and $F \circ T'_t \circ F^{-1}g_0$ satisfies this equation.
Consequently these two functions coincide implying \eqref{eqpropdualop} for the semigroups $T_t$ and $T_t^{D(f)}$, as required.
\end{proof}

Thus we have shown that under appropriate assumptions the $f$-dual operators to the first order and diffusion operators respectively
are again first order and diffusion operators respectively defining the $f$-dual or Pareto dual processes.

It is instructive to see which diffusions are self-dual. This is given by the following result that is a direct
consequence of Propositions \ref{propParetodualdiff} and \ref{propParetodualdeterm}.

\begin{prop}
\label{propParetoselfdualdiff}
Let
\begin{equation}
 \label{eqdiffParselfdual}
  L\phi(x)=\sum_{i,j=1}^d a_{ij}(x_i,x_j) \frac{\pa^2 \phi}{\pa x_i \pa x_j}(x)+\frac12 \sum_{j=1}^d a'_{jj}(x_j) \frac{\pa \phi}{\pa x_j}(x)
  \end{equation}
 with a positive definite (possibly not strictly) diffusion matrix $a(x)=(a_{ij}(x))$
 such that $a_{ij}$ depend only on $x_i,x_j$ and are continuously differentiable (with bounded derivatives).
 Then the diffusion generated by $L$ is self-dual in the Pareto sense.
 \end{prop}

\subsection{Application to other cones}

Generalization of our results to orders arising from cones $C(e_1, \cdots, e_d)$
can be obtained by the change of variables, though the calculations quickly become rather cumbersome.
Let us consider only the simple example of the two-dimensional cone \eqref{eqdeftwodimlightcone}.
The question we are going to answer is as follows: under what conditions the diffusion operator
\begin{equation}
\label{eqtwodimdiff}
Lg(x,y)=a(x,y)\frac{\pa ^2g}{\pa x^2} +2b(x,y)\frac{\pa ^2g}{\pa x \pa y}+c(x,y)\frac{\pa ^2g}{\pa y^2}
\end{equation}
generates a diffusion that has a dual in the sense of the order generated by $C$, and how the dual generator looks like.
Having in mind the relation with the standard Pareto order we can expect that the coefficients should depend in certain way on two arbitrary functions of one variable and one arbitrary function of two variables. This is in fact the case as the following result shows.

 \begin{prop}
\label{proplightconetwodimdualdiff}
Let $L$ of form \eqref{eqtwodimdiff} with smooth coefficients generate a Feller diffusion $X_t^x$.
If the coefficients have the form
 \begin{equation}
\label{eq1proplightconetwodimdualdiff}
\begin{aligned}
& a(x,y)=\al (x+y)+\be (x-y) +\om (x,y), \\
& c(x,y)=\al (x+y)+\be (x-y) -\om (x,y), \\
& b(x,y)=\al (x+y)-\be (x-y)
\end{aligned}
\end{equation}
with some smooth functions $\al, \be,\om$, then
$X_t^x$ has the dual diffusion $Y_t^y$ so that \eqref{eqdefdualfrompreorder} holds with $M=C(e_1,e_2)$ of form \eqref{eqdeftwodimlightcone},
where $Y_t^y$ is generated by the operator
 \begin{equation}
\label{eq2proplightconetwodimdualdiff}
L^{D(f)}g=Lg+4 (\al ' (x+y)+\be' (x-y)) \frac{\pa g}{\pa x}(x,y)+4 (\al ' (x+y)-\be' (x-y)) \frac{\pa g}{\pa y}(x,y).
\end{equation}
 \end{prop}

 \begin{proof} Formulas \eqref{eq1proplightconetwodimdualdiff} are obtained from Proposition \ref{propParetodualdiff}
 by rotation of coordinates, that is by change $x'=x+y, y'=x-y$.
 \end{proof}

\subsection{Duality from Pareto order: jump processes}

Let us now turn to the generators $L$ of pure jump processes, that is
\begin{equation}
\label{eqjumpgen}
L\phi(x)=\int_{\R^d} (\phi(w)-\phi(x))\nu (x,dw)
\end{equation}
with some bounded stochastic kernel $\nu$. For a measure $Q$ having a density with respect to Lebesgue measure, let us write shortly $L'q$ for the measure $L'Q$.
We have
\[
L'q(dz)=\int_{\R^d} q(x) \nu(x,dz) dx-q(z) dz \int_{\R^d} \nu (z, dw).
\]
Consequently, relabeling the variables of integration, we have
\[
F\circ L' (q)=(-1)^d \int_{z\ge y} (L'q)(dz)
\]
\[
=(-1)^d \int_{w\ge y} \int_{\R^d} q(z) \nu(z,dw) \, dz- (-1)^d \int_{z\ge y} \int_{\R^d} q(z) \nu(z,dw) \, dz.
\]
The integrals in the two terms partially cancel. Namely, we can write
\[
F\circ L' (q)=(-1)^d \int q(z) \left( \1_{z\ge y} \left[\int_{w\ge y} \nu (z,dw)-\int \nu (z,dw)\right]+\1_{z\ngeq y} \int_{w\ge y} \nu (z,dw)\right) \, dz,
\]
implying
\[
F\circ L' (q)=(-1)^d \int q(z) \left[\1_{z\ngeqslant y} \int_{w\ge y} \nu (z,dw)- \1_{z\ge y} \int_{w\ngeq y} \nu (z,dw)\right] \, dz.
\]
Hence, for a smooth ($d$ times differentiable) function $g$ we can write either
\[
L^{D(f)}g=F\circ L' \circ F^{-1} g (y)
\]
\begin{equation}
\label{eqjumpdual1}
=(-1)^d \int \frac{\pa^d g(z)}{\pa z_1 \cdots \pa z_d}
 \left[\1_{z\ngeq y} \int_{w\ge y} \nu (z,dw)- \1_{z\ge y} \int_{w\ngeqslant y} \nu (z,dw)\right] \, dz,
\end{equation}
or
\[
L^{D(f)}g=F\circ L' \circ F^{-1} g (y)
\]
\begin{equation}
\label{eqjumpdual2}
=(-1)^d \int_{w\ge y} \int_{\R^d}  \frac{\pa^d g(z)}{\pa z_1 \cdots \pa z_d} \nu(z,dw) \, dz
- (-1)^d \int_{z\ge y} \int_{\R^d}  \frac{\pa^d g(z)}{\pa z_1 \cdots \pa z_d} \nu(z,dw) \, dz.
\end{equation}

If $\nu (z,dw)$ depends smoothly on $z$, this expression can be rewritten by moving the derivatives from $g$ to $\nu$.
For this transformation expression \eqref{eqjumpdual2} is more handy than \eqref{eqjumpdual1}.
To perform the integration by parts in its second term we shall use the following simple formula
(with a straightforward proof by mathematical induction)
\begin{equation}
\label{eqbypartsmultidim}
\int_{z\ge y} \frac{\pa^d g(z)}{\pa z_1 \cdots \pa z_d} \phi (z) dz
=(-1)^d \sum _{I\subset \{1, \cdots, d\}} \int_{z_I \ge y_I} g(y_{\bar I}, z_I)
\frac{\pa^{|I|} \phi}{\pa z_I}(y_{\bar I}, z_I) dz_I,
\end{equation}
which is valid when the boundary terms at infinity vanish, for instance if either $\phi$ or $g$ vanish at infinity with all its derivatives.
Here $|I|$ is the number of indices in $I$, the integral over the set $\{z_I \ge y_I\}$ is $|I|$-dimensional and
$(y_{\bar I}, z_I)$ denotes the vector whose coordinates with indices from $I$ are those of the vector
 $z$ and other coordinates are from the vector $y$.

 Using this formula we transform \eqref{eqjumpdual2}
 into the expression
\[
L^{D(f)}g(y)=F\circ L' \circ F^{-1} g (y)
\]
\[
= \int_{w\ge y} \int_{\R^d} g(z) \frac{\pa^d \nu }{\pa z_1 \cdots \pa z_d} (z,dw) \, dz
-\sum _{I\subset \{1, \cdots, d\}} \int_{z_I \ge y_I} dz_I g(y_{\bar I}, z_I)
\int_{\R^d} \frac{\pa^{|I|} \nu}{\pa z_I}(y_{\bar I}, z_I, dw).
\]
Singling out from the sum the terms corresponding to $I$ being empty and $I$ being the whole set
$ \{1, \cdots, d\}$, this rewrites as
\[
\int_{w\ge y} \int_{\R^d} g(z) \frac{\pa^d \nu }{\pa z_1 \cdots \pa z_d} (z,dw) \, dz
-\int_{z\ge y} \int_{\R^d} g(z) \frac{\pa^d \nu }{\pa z_1 \cdots \pa z_d} (z,dw) \, dz
\]
\[
-{\sum}' _{I\subset \{1, \cdots, d\}} \int_{z_I \ge y_I} dz_I g(y_{\bar I}, z_I)
\int_{\R^d} \frac{\pa^{|I|} \nu}{\pa z_I}(y_{\bar I}, z_I, dw)
-g(y) \int_{\R^d} \nu (y,dw),
\]
where $\sum'$ denotes the sum over all proper subsets $I$, i.e. all subsets $I$ excluding
empty set and the whole set $ \{1, \cdots, d\}$.
Performing the cancelation in the first two terms yields finally
(see the trick leading to \eqref{eqjumpdual1})
\[
L^{D(f)}g(y)=F\circ L' \circ F^{-1} g (y)=-g(y) \int_{\R^d} \nu (y,dw)
\]
\[
-\sum' _{I\subset \{1, \cdots, d\}} \int_{z_I \ge y_I} dz_I g(y_{\bar I}, z_I)
\int_{\R^d} \frac{\pa^{|I|} \nu}{\pa z_I}(y_{\bar I}, z_I, dw)
\]
\begin{equation}
\label{eqjumpdual3}
+ \int_{\R^d} g(z) dz \left[\1_{z\ngeq y} \int_{w\ge y} \frac{\pa ^d\nu}{\pa z_1 \cdots \pa z_d} (z,dw)
  - \1_{z\ge y} \int_{w\ngeq y} \frac{\pa ^d\nu}{\pa z_1 \cdots \pa z_d} (z,dw)\right].
\end{equation}

For instance, for $d=1$
 \begin{equation}
 \label{eqourdualjumponedim}
  L^{D(f)}g(y)=\int_{-\infty}^y g(z) \, dz \int_{w\ge y} \frac{\pa \nu}{\pa z} (z, dw)
  -\int_y^{\infty} g(z) \, dz \int_{w<y}\frac{\pa \nu}{\pa z} (z, dw) - g(y) \int \nu (y,dw),
\end{equation}
which is the formula essentially obtained in \cite{Ko03} and \cite{Ko10},
and for $d=2$
 \[
  L^{D(f)}g(y)=-g(y_1,y_2) \int \nu (y,dw)
  \]
  \[
   -\int_{z_1\ge y_1} g(z_1,y_2) dz_1 \int \frac{\pa \nu}{\pa z_1}(z_1,y_2, dw)
-\int_{z_2\ge y_2} g(y_1,z_2) dz_2 \int \frac{\pa \nu}{\pa z_2}(y_1,z_2, dw)
\]
\begin{equation}
 \label{eqourdualjumptwodim}
  +\int g(z_1,z_2) dz_1dz_2 \left[\1_{z\ngeq y} \int_{w\ge y} \frac{\pa ^2\nu}{\pa z_1 \pa z_2} (z,dw)
  - \1_{z\ge y} \int_{w\ngeq y} \frac{\pa ^2\nu}{\pa z_1 \pa z_2} (z,dw)\right].
  \end{equation}

\begin{remark}
It is worth stressing that one should be cautious in using these formulas as they may not be true for $f$ not vanishing at infinity,
say even for a constant function $f$ (so that these formulas cannot be used even for checking conservativity condition $L^{D(f)} \1=0$).
Generally one has to use the following extension of
\eqref{eqbypartsmultidim} (also proved by direct induction) that is valid whenever $g$, $\phi$ are smooth and such that for
all $I\subset \{1, \cdots ,d\}$ and $y_{\bar I}$ there exist finite limits of the functions $g(y_{\bar I}, z_I)$, $\phi (y_{\bar I}, z_I)$
and their derivatives in $z_I$, as $z_I\to \infty$ (here $\infty$ means precisely $+\infty$):

\[
\int_{z\ge y} \frac{\pa^d g(z)}{\pa z_1 \cdots \pa z_d} \phi (z) dz
\]
\begin{equation}
\label{eqbypartsmultidimgen}
=(-1)^d \sum _{I\subset \{1, \cdots, d\}} \int_{z_I \ge y_I}
\left[ \sum_{J\subset \bar I}  (-1)^{|J|} g(y_{\bar I \setminus J}, \infty_J, z_I)
\frac{\pa^{|I|} \phi}{\pa z_I}(y_{\bar I \setminus J}, \infty_J, z_I)\right] dz_I,
\end{equation}
where $(y_{\bar I \setminus J}, \infty_J, z_I)$ denotes the vector with $\bar I \setminus J$ -coordinates from $y$, $I$-coordinates from $z$ and other coordinates being $+\infty$. For instance, in case $d=2$ we have
\[
\int_{y_1}^{\infty}\int_{y_2}^{\infty} \frac{\pa^2 g(z)}{\pa z_1 \pa z_2} \phi (z) \, dz
=\int_{y_1}^{\infty}\int_{y_2}^{\infty} \frac{\pa^2 \phi (z)}{\pa z_1 \pa z_2} g (z) \, dz
\]
\[
+\int_{y_1}^{\infty} \left[ g(z_1, y_2) \frac{\pa^2 \phi}{\pa z_1 } (z_1, y_2)- g(z_1, \infty) \frac{\pa^2 \phi}{\pa z_1 } (z_1, \infty)\right] dz_1
\]
\[
+\int_{y_2}^{\infty} \left[ g(y_1, z_2) \frac{\pa^2 \phi}{\pa z_2 } (y_1, z_2)- g(\infty, z_2) \frac{\pa^2 \phi}{\pa z_2 } (\infty, z_2)\right] dz_2
\]
\begin{equation}
\label{eqbypartsmultidimgen}
+g(y_1,y_2)\phi(y_1,y_2)-g(\infty,y_2)\phi(\infty,y_2)-g(y_1,\infty)\phi(y_1,\infty)+g(\infty,\infty)\phi(\infty,\infty).
\end{equation}
\end{remark}

Assuming that for all $y$
\begin{equation}
\label{eqcondonedimduinf}
\lim_{z\to -\infty} \int_{w\ge y} \nu (z, dw) =0, \quad \lim_{z\to \infty} \int_{w < y} \nu (z, dw) =0,
\end{equation}
equation \eqref{eqourdualjumponedim} rewrites in the equivalent conservative form
 \begin{equation}
 \label{eqourdualjumponedim1}
  L^{D(f)}g(y)=\int_{-\infty}^y (g(z)-g(y)) \, dz \int_{w\ge y} \frac{\pa \nu}{\pa z} (z, dw)
  -\int_y^{\infty} (g(z)-g(y)) \, dz \int_{w<y}\frac{\pa \nu}{\pa z} (z, dw).
\end{equation}

\begin{prop}
\label{propParetodualjump}
Let $L$ have form \eqref{eqjumpgen} with a bounded weakly continuous stochastic kernel $\nu$, so that
 $L$ generates a $C$-Feller (i,e. its semigroup preserves continuous functions) jump process in $\R^d$ that we denote $X_t^x$.
Then $L^{D(f)}$ is given by \eqref{eqjumpdual1}. If the kernel $\nu$ has continuous bounded mixed derivatives,
so that
\[
 \frac{\pa ^{|I|}\nu}{\pa z_I}(z,dw)
 \]
 is again a bounded kernel (possibly signed) for any nonempty subset $I\in \{1, \cdots d\}$
 (including the whole set $\{1, \cdots d\}$), then
$L^{D(f)}$ can be rewritten as \eqref{eqjumpdual3}. Finally $L^{D(f)}$ generates itself a $C$-Feller Markov process
that we denote $Y_t^y$ if and only if the following conditions hold:

All mixed derivatives of orders from $1$ to $d-1$ of the jump rates are non-positive, i.e.
\begin{equation}
\label{eq1propParetodualdiff}
\int_{\R^d} \frac{\pa ^{|I|}\nu}{\pa z_I}(z, dw) \le 0
\end{equation}
for any proper subset $I$ of $\{1, \cdots d\}$;
and
\begin{equation}
\label{eq2propParetodualdiff}
\begin{aligned}
& \int_{w\ge y}  \frac{\pa ^d\nu}{\pa z_1 \cdots \pa z_d} (z,dw) \ge 0, \quad z \ngeq y, \\
& \int_{w\ngeq y}  \frac{\pa ^d\nu}{\pa z_1 \cdots \pa z_d} (z,dw) \le 0, \quad z \ge y.
\end{aligned}
\end{equation}

If this is the case, the duality relation
\eqref{eqdefdualfromParetoorder} holds.
 \end{prop}

 \begin{proof}
 Everything is proved  apart from the criterion for the generation of a Markov process.
 To get it one only has to note that the operator $\int g(z) \mu (y, dz) -\al (y) g(y)$ with given kernel $\mu$ and function $\al$
 is conditionally positive (and generates a process) if and only if the kernel $\mu $ is stochastic (i.e. positive),
 and that the kernels from various terms in \eqref{eqjumpdual3} are mutually singular, so that this positivity condition
 should be applied separately to each term.

One completes the proof by the same argument as used at the end of the proof of Proposition \ref{propParetodualdiff}.
 \end{proof}

 Couple of remarks are in order here.
 Condition \eqref{eq2propParetodualdiff} is not very transparent. A simple particular case
 to have in mind is when the kernel $\nu$ decomposes into a sum of kernels depending on all variables but for one, i.e.
 \[
 \nu (z, dw) =\sum_j \nu_j(z_1, \cdots, z_{j-1}, z_{j+1}, \cdots , z_d, dw),
 \]
in which case the condition \eqref{eq2propParetodualdiff} becomes void (thus trivially satisfied).
On the other hand, conditions \eqref{eq1propParetodualdiff} are easy to check. To visualize this condition it is instructive to observe
that if $q$ is a density of a positive measure on $\R^d$, then the distribution function
\[
g(x) =\int_{z\ngeq x} q(z) dz
\]
is positive, but has all mixed derivatives negative. Even more specifically, if $\nu$ decomposes into a sum of kernels depending on one variable only,
that is
\[
 \nu (z, dw) =\sum_j \nu_j(z_j, dw),
 \]
all conditions of Proposition \ref{propParetodualjump} are reduced to an easy to check requirement that all rates $\int \nu_j(z_j,dw)$
are decreasing functions of $z_j$.

Let us note that the method of the calculation of dual used above can still be used for processes with a boundary.
For instance, let us consider a process on $\R_+$ with the generator
\begin{equation}
\label{eqjumpgenonedimb}
L\phi(x)=\int_{\R_+} (\phi(w)-\phi(x))\nu (x,dw).
\end{equation}
The operator $L'$ takes the form
\[
L'q(dz)=\int_{\R_+} q(x) \nu(x,dz) dx-q(z) dz \int_{\R_+} \nu (z, dw)
\]
and the same calculations as above yield
 \[
  L^{D(f)}g(y)=\int_0^y g(z) \, dz \int_{w\ge y} \frac{\pa \nu}{\pa z} (z, dw)
  -\int_y^{\infty} g(z) \, dz \int_{0\le w<y}\frac{\pa \nu}{\pa z} (z, dw)
  \]
  \begin{equation}
 \label{eqourdualjumponedimb}
   - g(y) \int \nu (y,dw) +g(0) \int_{w\ge y} \nu (0, dw),
\end{equation}
that is, an additional term appears arising from additional boundary taken into account while integrating by parts.
Under assumption \eqref{eqcondonedimduinf}, this rewrites in the equivalent conservative form
 \[
  L^{D(f)}g(y)=\int_0^y (g(z)-g(y)) \, dz \int_{w\ge y} \frac{\pa \nu}{\pa z} (z, dw)
  \]
  \begin{equation}
 \label{eqourdualjumponedimb1}
  -\int_y^{\infty} (g(z)-g(y)) \, dz \int_{0\le w<y}\frac{\pa \nu}{\pa z} (z, dw) +\int_{w\ge y} (g(0)-g(y)) \nu (0, dw).
\end{equation}

We assume strong smoothness condition for $\nu$, which forces the dual L\'evy kernel to have a density.
This is not necessary. Just assuming monotonicity of $\int_{w\ge y} \nu (z, dw)$ and $\int_{w<y} \nu (z, dw)$
(and thus the existence almost sure of non-negative derivatives of these functions of $z$),
we obtain, instead of \eqref{eqourdualjumponedim1}, the formula
\begin{equation}
 \label{eqourdualjumponedim1}
  L^{D(f)}g(y)=\int_{-\infty}^y (g(z)-g(y)) \, d_z \int_{w\ge y} \nu (z, dw)
  -\int_y^{\infty} (g(z)-g(y)) \, d_z \int_{w<y} \nu (z, dw),
\end{equation}
with similar modifications for \eqref{eqourdualjumponedimb1} and analogously for $d$-dimensional case.

Let us mention the link with the theory of stochastic monotonicity.
A Markov process $X_t^x$ is called stochastically monotone with respect to Pareto ordering if the function
$\P(X_t^x \ge y)$ is a monotone function of $x$ for any $y$. Stochastic monotonicity is studied for various classes of processes,
see \cite{ChWa}, \cite{Chenbook04}, \cite{Ko03}, \cite{Kobook11}, \cite{Zh}, \cite{JMW}, \cite{Rabe} and references therein.
If duality \eqref{eqdefdualfromParetoorder} holds, then
$X_t^x$ is obviously stochastically monotone, but, generally speaking,  this condition is too weak to ensure duality, because stochastic
monotonicity of a positive function on $\R^d$ does not imply (apart from one-dimensional case) that it
is the multi-dimensional distribution function for some positive measure. Therefore it is remarkable enough that for diffusion processes with generators
\eqref{eqdiffgen} the conditions of stochastic monotonicity and of the existence of Pareto dual coincide. Even for deterministic processes this is already not so,
as for stochastic monotonicity of processes generated by operators \eqref{eqdetermgen}, $b_j$ are allowed to depend on other coordinates $x_k$ (in a monotone way, see e.g. \cite{ChWa} and references therein to previous works).
Stochastic monotonicity and related duality are well developed for Markov chains, see e.g. \cite{And} and \cite{VanDo},
for birth and death processes and one-dimensional diffusions see \cite{CoRo}.

We assumed boundedness of all coefficients involved. This simplification leads to the most straightforward formulations
that catch up the essence of duality. Of course, extensions to unbounded kernel rates, diffusion coefficients, etc, are possible
under the conditions that ensure that all processes involved are well defined.

\subsection{Arbitrary Feller processes}

We have analyzed three classes of the generators $L$ separately. But it is clear that if we consider a process with the generator
being the sum of the generators of different classes, then applying conditions of the results above to each term separately
will ensure that the dual to the sum is also conditionally positive and generates a process
leading to the duality relation \eqref{eqdefdualfromParetoorder}.
For simplicity, we shall give the corresponding result for one-dimensional Feller processes,
but extension to higher dimensions is straightforward. For this case, the generators of the dual were obtained in \cite{Ko10}
(which contains an annoying systematic typo with the wrong sign $'-'$ before the second term of \eqref{eqourdualdiff1}) by approximating continuous
 state space generators by discrete Markov chains. The method of the present paper will give the same result without any technical
restrictions used in \cite{Ko10} this yielding the complete characterization.

\begin{prop}
\label{propParetodualonedimful}
Let a Feller process $X_t^x$ in $C_{\infty}(\R)$ have a generator
\begin{equation}
\label{eq1Paretodualonedimful}
Lg(x)=a(x) \frac{d^2}{dx^2} g(x)+b(x) \frac{d}{dx} g(x)+ \int_{-\infty}^{\infty} (g(z)-g(x)-(z-x)g'(x)\1_{|z-x|\le 1})\nu (x, dz)
\end{equation}
with $a, b\in C^2(\R)$, $a$ being non-negative, and with the weakly continuous L\'evy kernel $\nu$ such that, for any $y$,
conditions \eqref{eqcondonedimduinf} hold
 and the functions
\begin{equation}
\label{eq2Paretodualonedimful}
\int_{w\ge y} \nu (z, dw), \quad -\int_{w< y} \nu (z, dw)
\end{equation}
are non-decreasing in $z$, for $z<y$ and $z>y$ respectively, so that their derivatives exist almost surely and are non-negative.
Moreover
 \begin{equation}
\label{eq3Paretodualonedimful}
\1_{z<y} d_z \int _{w\ge y} \nu (z,dw) +\1_{z>y} d_z \int _{w < y} \nu (z,dw)
\end{equation}
is a L\'evy kernel (it integrates $\min (1,(w-z)^2)$ and the integral
 \[
  \int_{y-1}^{y+1} (z-y)\left[\1_{z<y}(\nu (y, dz)+d_z \int _{w\ge y} \nu (z,dw))
 + \1_{z>y}(\nu (y, dz)-d_z \int _{w < y} \nu (z,dw))\right]
 \]
 exists, at least in the sense of the main (or the Cauchy) value.
Then the dual process $Y_t^y$ exists (in the sense of \eqref{eqdefdualfromParetoorder}) and has the generator
\[
Lg(y)=a(y) \frac{d^2}{dy^2} g(y)+(a'(y)-b(y)) \frac{d}{dy} g(y) 
\]
\[
+\int_{-\infty}^y (g(z)-g(y)-(z-y)g'(y)\1_{|z-y|\le 1})d_z \left(\int _{w\ge y} \nu (z,dw)\right)
 \]
 \[
 -\int_y^{\infty} (g(z)-g(y)-(z-y)g'(y)\1_{|z-y|\le 1})d_z \left(\int _{w < y} \nu (z,dw)\right)
\]
\begin{equation}
\label{eq1Paretodualonedimful}
 +g'(y) \int_{y-1}^{y+1} (z-y)\left[\1_{z<y}(\nu (y, dz)+d_z \int _{w\ge y} \nu (z,dw))
 + \1_{z>y}(\nu (y, dz)-d_z \int _{w < y} \nu (z,dw))\right]
\end{equation}
 \end{prop}

 \begin{proof}
 Formula \eqref{eq1Paretodualonedimful} is obtained by combining \eqref{eqourdualjumponedim1}, \eqref{eqourdualdiff1}
 and \eqref{eqourdualdeterm}. Conditions given ensure that the dual operator is well defined as a L\'evy-Khintchin type operator with variable coefficients.
 \end{proof}

 \begin{remark}
 As shown in \cite{Ko10} and Theorem 5.9.2 of \cite{Kobook11}, conditions of stochastic monotonicity (monotonicity of functions
 \eqref{eq2Paretodualonedimful}) are sufficient for
 the operator \eqref{eq1Paretodualonedimful} to generate a Feller process, so that this condition can be dispensed with.
 \end{remark}

 As a corollary of Proposition \ref{propParetodualonedimful}, we can get now the full characterization of self - duality.
 \begin{prop}
\label{propParetoselfdualonedimful}
Let a Feller process $X_t^x$ in $C_{\infty}(\R)$ have a generator \eqref{eq1Paretodualonedimful}.
Then it is self dual (in the sense of \eqref{eqdefdualfromParetoorder}) if and only if the following conditions holds:
 \begin{equation}
\label{eq1propParetoselfdualonedimful}
b(x)=a'(x)/2, \quad d_y \nu (y,dz)+d_z \nu (z, dy)=0.
\end{equation}
In particular, if $\nu$ has a density $\nu (z,w)$, which is differentiable with respect to the first argument, then the second equation
 of \eqref{eq1propParetoselfdualonedimful} rewrites as
  \begin{equation}
\label{eq2propParetoselfdualonedimful}
\frac{\pa \nu}{\pa y} (y,z)+\frac{\pa \nu}{\pa z} (z,y)=0.
\end{equation}
Clearly, this condition is satisfied for $\nu (y,z) =g(|y-z|)$ with a smooth $g$, which corresponds to symmetric L\'evy generators.
\end{prop}

\begin{proof}
The condition on $b$ follows from Proposition \ref{propParetoselfdualdiff}. The condition on $\nu $ arises by the comparison of
the integral terms of \eqref{eq1Paretodualonedimful} with \eqref{eq1Paretodualonedimful} separately for $y>z$ and $y<z$.
\end{proof}

\section{Stochastic $f$-duality from translation invariant $f$}
\label{secdualfromLevygen}

We have analyzed in some detail the duality arising from Pareto ordering. In general case explicit calculations are not always available.
However, we shall propose here some general scheme for the analysis of translation-invariant $f$, that is $f$ depending only on the
difference of their arguments:
\[
f(x,y)=f(y-x),
\]
with some other function $f$ that we still denote by $f$ (with some ambiguity).

Thus the operator $F$ from \eqref{eqdefdualsemigroupsF} when applied  to a measure $Q$ with density $q$ takes the form
\begin{equation}
\label{eqdefdualsemigroupsFsep}
g(y)=(FQ)(y)=\int_{\R^d} f(y-x) q(dx),
\end{equation}
i.e. it becomes a convolution operator.
It is then well known that under appropriate regularity assumptions, $f$ is the fundamental solution of the pseudo-differential operator
$L_f$ with the symbol
\begin{equation}
\label{eqfundsolFourier1}
L_f(p)=\frac{1}{\hat f(p)},
\end{equation}
where
\[
\hat f(p)=\int e^{-ixp} f(x) dx
\]
is the Fourier transform of $f$.

\begin{remark}
In fact, by the definition of the fundamental solution,
\[
L_f \left(\frac{1}{i} \frac{\pa}{\pa x}\right) f(x) =\de (x),
\]
which by taking the Fourier transform from both sides rewrites as
\[
L_f(p)\hat f (p)=1,
\]
as claimed.
\end{remark}

Hence $g(y)$ from \eqref{eqdefdualsemigroupsFsep} solves the equation $L_fg=q$, so that $F^{-1}=L_f$.
Of course, for an arbitrary $f$, the operator $L_f$ can be quite awkward and the identification of the appropriate classes of functions $q,g$ quite nontrivial.
Let us consider the case when everything is well understood, namely the case of $L_f$ being a Laplacian, or more generally, its fractional power.

It is well known that the fundamental solution for the Laplace operator $\De$ in dimension $d\ge 3$ is the function
\[
f(x)=-\frac{1}{(d-2)\si_{d-1}} \frac{1}{|x|^{d-2}},
\]
where $\si_{d-1}$ is the area of the unit sphere in $\R^d$. Hence the dual operator \eqref{eqpropdualop} takes the form
\begin{equation}
\label{eqpropdualopLap}
T^{D(f)}=\De^{-1} \circ T' \circ \De,
\end{equation}
and the generator for the corresponding dual semigroup becomes
 \begin{equation}
\label{eqpropdualopLapgen}
L^{D(f)}=\De^{-1} \circ L' \circ \De.
\end{equation}
Let $L$ be a diffusion operator of the special kind:
\[
Lg(x)=a(x)\De g(x)
\]
with a nonnegative bounded smooth function $a(x)$. Then $L'=\De \circ a(x)$ and thus
 \begin{equation}
\label{eqpropdualopLapgen}
L^{D(f)}=\De^{-1} \circ L' \circ \De=L,
\end{equation}
so that $L$ is self $f$-dual.

Noting that in dimensions $d=2$ the fundamental solution for the Laplacian is known to be $\log |x|/2\pi$
 we get the following.

\begin{prop}
\label{propdualbyLap}
Let $X_t^x$ be the Feller diffusion generated by the operator $Lg(x)=a(x)\De g(x)$ in $\R^d$
with a nonnegative bounded smooth function $a(x)$. Then, for all $x,y \in \R^d$, we have
\begin{equation}
\label{eq1propdualbyLap}
 \E \frac{1}{|X_t^x-y|^{d-2}}=\E \frac{1}{|X_t^y-x|^{d-2}},
 \end{equation}
 \begin{equation}
\label{eq2propdualbyLap}
 \E \log |X_t^x-y|=\E \log |X_t^y-x|, \quad d=2
 \end{equation}
 for $d\ge 3$ and $d=2$ respectively.
 \end{prop}

 Turning to the fractional Laplacian $|\De|^{\al/2}$ in $\R^d$ with $\al \in (0,2)$, $d\ge 2$, let us recall that the inverse operator is given by the so-called Riesz potential
 \[
 |\De|^{-\al/2}g(x)= I^{\al}g(x)= \frac{1}{H_d(\al)} \int_{\R^d} \frac{g(y)\, dy}{|x-y|^{d-\al}},
 \]
 where
 \[
 H_d(\al)=2^{\al} \pi ^{d/2} \frac{\Ga (\al/2)}{\Ga ((d-\al)/2)},
 \]
 see e.g. \cite{Hel}.
 Hence, the operator $|\De|^{\al/2}$ is $L_f$ for
 \[
 f(x)=\frac{1}{H_d(\al)}\frac{1}{|x|^{d-\al}}.
 \]

 Let us consider a stable-like process generated by the operator
 \[
Lg(x)=-a(x)|\De|^{\al/2} g(x)
\]
with a positive smooth function $a(x)$.
Then $L'=|\De|^{\al/2} \circ a(x)$ and thus
 \begin{equation}
\label{eqpropdualopLapgen}
L^{D(f)}=|\De|^{-\al/2} \circ L' \circ |\De|^{\al/2}=L,
\end{equation}
so that $L$ is self $f$-dual. Thus we proved the following extension of Proposition \ref{propdualbyLap}:

\begin{prop}
\label{propdualbyLapfr}
Let $X_t^x$ be the stable-like process generated by the operator $Lg(x)=a(x)|\De|^{\al/2} g(x)$ in $\R^d$ with  $d\ge 2$, $\al \in (0,2]$ excluding the case
$d=\al=2$ (for which \eqref{eq2propdualbyLap} holds), and
with a nonnegative bounded smooth function $a(x)$. Then,  for all $x,y \in \R^d$,
\begin{equation}
\label{eqpropdualbyLapfr}
 \E \frac{1}{|X_t^x-y|^{d-\al}}=\E \frac{1}{|X_t^y-x|^{d-\al}}.
 \end{equation}
 \end{prop}


\section{Stochastic duality for processes in $\bar \R_+$}
\label{secdualbound}

\subsection{Reflected and absorbed diffusions in $\bar \R_+$}

We shall deduce some consequences from our general approach to processes on $\R_+$ that are dual in the sense \eqref{eqdefstanddual2ndmon}.
 $C_{\infty}^k(\R^d)$ will denote the space of $k$ times differentiable functions on $\R^d$ with all these derivatives vanishing at infinity.
 $C_{\infty}^k(\bar \R_+)$ is the restriction of functions from  $C_{\infty}^k(\R)$ on $\bar \R_+=\{x\ge 0\}$.

Consider a Feller process $X_t^x$ on $\R$ generated by operator \eqref{eq1Paretodualonedimful} under the conditions of Proposition
\ref{propParetodualonedimful} assuming additionally that

(A) $a\in C^2(\R)$ and is an even function such that $a(x) \ge 0$, $b\in C^2(\R)$ and is an odd function (implying $b(0)=0$),
the support of $\nu$ is in $\R_+$ for $x\ge 0$ and $\nu (-x,dy)=R\nu (x,dy)$, where $R$ denotes the reflection of
the measure with respect to the origin (so that, by definition,  $\int \phi (y) R\nu (x, dy)=\int \phi (-y) \nu (x, dy)$).

Then, as is well known, see e.g. Theorem 6.8.1 in \cite{Kobook11},
the magnitude $|X_t^x|$ is itself a Markov process on $\R_+$, also referred to as $X_t^x$ reflected at the origin.
Moreover, if the transition probabilities of $X_t^x$ are $p_t(x,dy)$, then
 $|X_t^x|$ has the transition density
 \[
 p^{ref}_t(x,dy)=p_t(x,dy)+R p_t(x, dy),
 \]
  and the semigroup $T_t^{ref}$ of $|X_t^x|$ can be obtained from the semigroup $T_t$ of  $X_t^x$ by the restriction to even functions.

\begin{remark}
(i) Assuming that the kernel $\nu$ is twice smooth would imply that the space $C^2_{\infty}(\R)$ is an invariant core for $X_t^x$ and consequently
  that the subspace of functions $f$ from  $C_{\infty}^2(\bar \R_+)$ such that $f'(0)=0$ is an invariant core for $|X_t^x|$.
(ii) If $X_t^x$ were a diffusion, the process $|X_t^x|$ on $\bar \R_+$
  would be stochastically monotone by the coupling argument, see e.g. Sect II,2 of \cite{Lig}) and hence by Siegmund's theorem \cite{Sieg}
 it had a Markov dual $Y_t^y$ on $\bar \R_+$ (in the sense \eqref{eqdefstanddual2ndmon})
 with absorbtion at the origin. In our case monotonicity follows from the construction of the dual below, which turns out to be
 given by a semigroup with a conditionally positive generator.
\end{remark}

 \begin{prop}
 \label{propdualdifRplus}
 Under the conditions of Proposition
\ref{propParetodualonedimful}, assumption (A) above and finally assuming that the measure $\nu (0, dw)$ is bounded,
 the dual process $Y_t^y$ is a Feller on $\bar \R_+$ absorbed at the origin and generated by the operator

 \[
L^Dg(y)=a(y) \frac{d^2}{dy^2} g(y)+(a'(y)-b(y)) \frac{d}{dy} g(y)+\int_{w\ge y} (g(0)-g(y) \nu (0, dw)
\]
\[
+\int_0^y (g(z)-g(y)-(z-y)g'(y)\1_{|z-y|\le 1})d_z \left(\int _{w\ge y} \nu (z,dw)\right)
 \]
 \[
 -\int_y^{\infty} (g(z)-g(y)-(z-y)g'(y)\1_{|z-y|\le 1})d_z \left(\int _{w < y} \nu (z,dw)\right)
\]
\begin{equation}
\label{eq1Paretodualonedimfulb}
 +g'(y) \int_{y-1}^{y+1} (z-y)\left[\1_{z<y}(\nu (y, dz)+d_z \int _{w\ge y} \nu (z,dw))
 + \1_{z>y}(\nu (y, dz)-d_z \int _{w < y} \nu (z,dw))\right]
\end{equation}
The semigroup $T_t^D$ of $Y_t^y$ is given explicitly by the formula
 \begin{equation}
 \label{eq2propdualdifRplus}
 (T^D_tg)(y)=g(0) \int_y^{\infty} p^{ref}_t(0,dz) +\int_0^{\infty} g(x) \left( \int_y^{\infty} \frac{\pa}{\pa x} p_t^{ref}(x,dz) \right) dx.
 \end{equation}

 \end{prop}

 \begin{proof}
 Using \eqref{eqpropdualopesemig} with $F^{-1}g(x)=-g'(x)$ we get for $g\in C_{\infty}^1(\bar \R_+)$
  \begin{equation}
 \label{eq3propdualdifRplus}
 (T^D_tg)(y)=-\int_y^{\infty} dz \int_0^{\infty} g'(x) p_t^{ref} (x,dz) \, dx,
 \end{equation}
 and hence
  \begin{equation}
 \label{eq4propdualdifRplus}
 (T^D_tg)(y)= g(0) \int_y^{\infty} p_t^{ref} (0,dz) +\int_0^{\infty} dx \int_y^{\infty} g(x) \frac{\pa}{\pa x} p_t^{ref} (x,dz),
 \end{equation}
 yielding \eqref{eq2propdualdifRplus} as required.

 It is worth stressing that this formula implies the conservativity condition $T^D_t \1=\1$ (preservation of constants by $T^D_t$), because
 \[
 \lim_{x\to \infty} \int_y^{\infty} p_t^{ref} (x,dz) =1
 \]
 by the Feller property and hence
  \begin{equation}
 \label{eq5propdualdifRplus}
 \int_0^{\infty}  \frac{\pa}{\pa x} \left( \int_y^{\infty} p_t^{ref}(x,dz) \right) dx=1- \int_y^{\infty} p_t^{ref}(0,dz).
 \end{equation}

Operators $T_t^D$ form a semigroup by Proposition \ref{propdualsemi}.
The form of the generator follows from \eqref{eqourdualjumponedimb1}.
As it is conditionally positive, the semigroup $T^D_t$ preserves positivity
and preserves constants thus being a semigroup of a Markov process. Moreover, as also seen directly from \eqref{eq2propdualdifRplus}, $T^D_tf(0)=f(0)$,
so that the value at the origin is preserved meaning that this process is absorbing at the origin.
\end{proof}

\begin{remark}
(i) Formula \eqref{eq3propdualdifRplus} is valid only for $g$ vanishing at infinity, and \eqref{eq2propdualdifRplus} extends it (yields a minimal extension)
to bounded functions on $\bar \R_+$. Plugging $g=1$ into \eqref{eq3propdualdifRplus} yields zero, not $1$. (ii) The attempt to use integration in \eqref{eq5propdualdifRplus} in opposite direction, at least when $p_t(x,dz)$ has a density $p_t(x,z)$,
and using $\lim_{x\to \infty} p^{ref}_t(x,z)=0$ would give
\[
\int_y^{\infty} dz \left( \int_0^{\infty}  \frac{\pa}{\pa x} p_t^{ref}(x,z) \, dx \right)=- \int_y^{\infty} p_t^{ref}(0,z) \, dz,
\]
which is different from the r.h.s. of \eqref{eq5propdualdifRplus}.
\end{remark}

It is worth noting additionally that if $a(0)\neq 0$ and $\nu =0$, then the subspace of functions $g$ from  $C_{\infty}^2(\bar \R_+)$
such that $g''(0)=0$ is an invariant core for $Y_t^y$. In fact, the condition $L^Dg(0)=0$ (following from $T^D_tg(0)=g(0)$) implies $g''(0)=0$.
On the other hand, if $a(0)=0$ and $\nu =0$, then $a(x)=ax^2(1+o(1)), b(x)=bx(1+o(1))$ as $x\to 0$ with $a\ge 0, b\in \R$ implying that $0$ is an unaccessible
boundary point, so that $X_t^x=|X_t^x|$ for $x>0$. In this case nothing comes out of the origin, so that
$p^{ref}_t(0,z)=0$ for all $z>0$ implying that the first term on the r.h.s. of \eqref{eq2propdualdifRplus} vanishes and
hence that $0$ is also unaccessible for $Y_t^y$ (which follows also from its generator). In particular, if additionally
$b(x)=a'(x)/2$, the process $|X_t^x|$ is self-dual on $\R_+$.

There is an extensive literature on the absorption - reflection link presented in Proposition \ref{propdualdifRplus}, mostly because
 of its natural interpretation in terms of ruin probabilities having important applications in insurance mathematics.  For piecewise deterministic
   Markov processes it was obtained in \cite{AsPe} (see also \cite{As98}) and used effectively in \cite{Dje93} for assessing ruin probabilities via large deviations.
   Then it was extended to diffusions with jumps in \cite{SigRy}, and to L\'evy processes in \cite{AsPi}. Our result is an extension
   of the corresponding result from \cite{SigRy}, as we do it for arbitrary stochastically monotone processes. Our proof is quite different,
   as it is more elementary, using effectively only formula  \eqref{eqpropdualopesemig}.

\subsection{Second dual and regularized dual}

Extension of the previous result to processes with a boundary from the right or with two boundaries is if course natural,
see \cite{AsPi}, but not
quite straightforward. We shall clarify the aspects of duality (even the definition has to be modified), needed for these cases
reducing our attention to diffusions just for simplicity

It is natural to ask whether the second dual coincides with the original process. For diffusions on $\R^d$ this is in fact the case, as is seen from Proposition
  \ref{propParetodualdiff}. However, for processes on $\R_+$ this dies not hold, as seen already from L\'evy's example of reflected Brownian motion. In fact,
  reflected BM cannot be dual to absorbing BM, as any dual process on $\R_+$ should be absorbing at the left end, that is at the origin, as seen directly
  from \eqref{eqdefstanddual2ndmon}. However, the reflected BM is 'almost dual' to the absorbing BM in the sense that
  $\P (Y_t^y \le x)=\P (X_t^x \ge y)$ (with $Y$ reflected and $X$ absorbing BM) holds for all $y\neq 0$ and all $x$. This suggests that the usual
  definition of duality imposes unnatural restrictions on the boundary. Consequently we shall give the following definition.
  Let $X_t^x$ be a stochastically monotone process on $[a, \infty)$ such that $\P (X_t^x \ge y)$ is right continuous in $x$. A process
  $Y_t^y$ on $[a, \infty)$ will be called a {\it regularized dual} to a process $X_t^x$ on $[a, \infty)$ if  \eqref{eqdefstanddual2ndmon} holds for all $x\ge a, y>a$, and
  the distribution for $y=a$ is defined by continuity as
  \begin{equation}
 \label{eqdefstanddual2ndmonreg}
\P (Y_t^a \le x)=\P (Y_t^{a_-} \le x)=\lim_{y\to a} \P (Y_t^y \le x).
\end{equation}

\begin{remark} (i) One could also relax the condition for $x=a$ defining $\P (Y_t^y \le a)=\lim_{x\to a} \P (Y_t^y \le x)$, but this would lead to the same
result, as for usual definition, due to the right continuity of $\P (X_t^x \ge y)$ in $x$. (ii) If one only assumes monotonicity
 of the function $\P (X_t^x \ge y)$, it would become natural to define the dual distribution $\P (Y_t^y \le x)$ as the right continuous modification
 of the function $\P (X_t^x \ge y)$.
\end{remark}

The following statement is now clear.

\begin{prop}
 \label{propdualdifRplussec}
 Under the assumptions of Proposition \ref{propdualdifRplus} the initial reflected process $|X_t^x|$ is a regularized dual to $Y_t^y$.
 Thus the second regularized dual to $|X_t^x|$ coincides with $|X_t^x|$.
 \end{prop}

 \begin{remark}
 The usual (not regularized) dual of $Y_t^x$ from  Proposition \ref{propdualdifRplus} is a rather pathological process $Z_t^z$, whose distributions
coincides with that of  $|X_t^z|$ for $z\neq 0$, but the origin is an unattainable point without escape from it. Thus $Z_t^z$ should be 'reflected from the origin'
without touching it.
\end{remark}

\begin{remark}
Of course one can deal with reflected processes on $\bar \R_-$ by introducing a symmetric notion of duality. Namely, for a process $X_t^x$ on an interval of $\R$ let us say that $Y_t^y$ is its {\it right dual},
if $\P (Y_t^y \le x)=\P (X_t^x \ge y)$ holds for all $x,y$ (that is, it is the usual duality used above) and {\it left dual} if $\P (Y_t^y < x)=\P (X_t^x > y)$
holds for all $x,y$, which is equivalent to $\P (Y_t^y \ge x)=\P (X_t^x \le y)$.
Thus, by definition, $Y_t^y$ is right dual to $X_t^x$ if and only if $X_t^x$ is left dual to $Y_t^y$. The theory of left dual processes on $\R_-$
(and their regularized version) is completely analogous to the theory of right dual process on $\R_+$.
\end{remark}

\end{document}